\documentclass[12pt,a4paper]{article}
\usepackage{amsfonts}
\usepackage{amsmath}
\usepackage{amssymb}
\usepackage{fullpage}
\usepackage{amsthm}

\newtheorem{theorem}{Theorem}[section]
\newtheorem{proposition}[theorem]{Proposition}
\newtheorem{lemma}[theorem]{Lemma} 
\newtheorem{corollary}[theorem]{Corollary}

\newtheorem{definition}[theorem]{Definition}

\newcommand{\cd}{{\mathrm{cd}\,}}
\newcommand{\nn}{{\mathbb N}}
\newcommand{\zz}{{\mathbb Z}}

\title{Subgroups of almost finitely presented groups}
\author{Ian J. Leary\thanks{Partially supported by a Research Fellowship 
from the Leverhulme Trust.  This work was done at MSRI, Berkeley, 
where research is supported by 
the National Science Foundation under Grant No.~DMS-1440140.}}

\date{\today}

\begin{document} 

\maketitle

\begin{abstract} 
We show that every countable group embeds in a group of type $FP_2$. 
\end{abstract}

\section{Introduction} 

In the late 1940's, Higman, Neumann and Neumann showed that every 
countable group embeds in a 2-generator group, in the same paper 
in which they introduced HNN-extensions~\cite{hnn}.  Neumann had 
already shown that there are uncountably many 2-generator groups, 
from which it follows that they cannot all embed in finitely presented 
groups~\cite{neumann}.  It was not until the early 1960's that Higman 
was able to characterize the finitely generated subgroups of 
finitely presented groups~\cite{higman}.  The  
Higman embedding theorem is a high-point of combinatorial group
theory that makes precise the connection between group presentations
and logic: it states that a finitely generated group $G$ embeds in some 
finitely presented group if and only if $G$ is recursively presented, 
i.e., there is an algorithm to write down the relations that hold 
in~$G$~\cite{higman}.  

A group $G$ is almost finitely presented\footnote{This definition was 
used by Bieri and Strebel~\cite{bieristrebel}; many authors use the 
phrase `almost finitely presented' for the weaker condition that the 
augmentation ideal in the mod-2 group algebra is finitely presented.}  
or $FP_2$ if its augmentation 
ideal $I_G$ is finitely presented as a module for its group algebra
$\zz G$ (see~\cite[VIII.5]{brown} or~\cite{bieristrebel} for more details).  
Every finitely
presented group is $FP_2$, and every $FP_2$ group is finitely
generated.  Bestvina and Brady gave the first examples of $FP_2$
groups that are not finitely presented~\cite{BB}, although these
examples arose as subgroups of finitely presented groups.
In~\cite{ufp} the author constructed groups of type $FP_2$ that do not
embed in any finitely presented group.  Given these examples it
becomes natural to look for an analogue of the Higman embedding
theorem for $FP_2$ groups.  Our main theorem answers this question.

\begin{theorem}\label{thm:main}
Every countable group embeds in an $FP_2$ group.  
\end{theorem} 

Although the statement is similar to the Higman-Neumann-Neumann
embedding theorem, the proof is much closer to the Higman embedding 
theorem.   In fact it is modelled on Valiev's proof of the 
Higman embedding theorem as described in~\cite[Sec.~IV.7]{lynsch},
which is a simplification of Valiev's first proof~\cite{valiev}. 
Our proof is simpler than these antecedents because we are 
not obliged to consider recursively enumerable sets.  We make the following 
definition, which is an analogue of Higman's notion of a benign
subgroup.  

\begin{definition} 
A subgroup $H$ of a finitely generated group $G$ is a 
{\sl homologically benign subgroup}  if the HNN-extension 
$$G_H=\langle G,t\colon t^{-1}ht=h\,\, h\in H\rangle$$ 
can be embedded in an $FP_2$ group.  
\end{definition} 

Theorem~\ref{thm:main} implies that all subgroups of finitely
generated groups are homologically benign, however showing that
various subgroups are homologically benign plays a major role in the
proof of Theorem~\ref{thm:main}.  The result below details what we
need from~\cite{ufp}; after the statement we outline how to deduce it
from results stated in~\cite{ufp}.

\begin{theorem}\label{thm:fromufp} 
For any fixed $l\geq 4$ and any set $S$ of integers with $0\in S$,
there is an $FP_2$ group $J=J(l,S)$ and a sequence $j_1,\ldots,j_l$ of
elements of $J$ such that $j_1^sj_2^s\cdots j_l^s=1$ if and only if
$s\in S$.
\end{theorem}

\begin{proof} 
The groups $G_L(S)$ that are constructed in~\cite{ufp} depend on a
connected flag simplicial complex $L$ and a set $S\subseteq \zz$.  If
$L$ has perfect fundamental group and contains an edge loop of length
$l$ that is not homotopic to a constant map, then $J=G_L(S)$ has the
claimed properties.  See~\cite[section~2]{ufp} for an explicit example
of a suitable $L$ in the case $l=4$; examples for larger $l$ can be
obtained by taking subdivisions of this $L$.

We expand a little by giving the precise results within~\cite{ufp}
that guarantee the various properties of the group~$J=J(l,S)$.  When
$0\in S\subseteq \zz$, \cite[theorem~1.2]{ufp} gives a presentation for
$G_L(S)$ with generators the directed edges of~$L$.
By~\cite[theorem~1.3]{ufp}, the group $G_L(S)$ is $FP_2$ if and only if
the fundamental group of $L$ is perfect.  If $j_1,\ldots,j_l$ is a
directed loop in $L$ that does not bound a disk then
by~\cite[Lemma~14.4]{ufp}, the word $j_1^sj_2^s\cdots j_l^s$ in the 
given generators for $J$ is equal to
the identity if and only if $s\in S$.
\end{proof}

Theorem~\ref{thm:fromufp} enables one to encode arbitrary subsets of
the natural numbers $\nn$ in presentations for $FP_2$ groups.  This
theorem replaces those parts of Valiev's proof that concern
Diophantine equations or those parts of Higman's proof that concern
recursive functions, each of which is used to encode recursively
enumerable subsets of $\nn$ in finite presentations.

\section{The proofs} 

Since this section is closely modelled on Lyndon and Schupp's 
account of the Higman embedding theorem~\cite[Sec.~IV.7]{lynsch}, 
we have tried to stay close to the notation that they use.  We 
also omit arguments that are identical to those in~\cite{lynsch}.  

Since we will be working with presentations, it is convenient to 
have a characterization of the $FP_2$ property in terms of presentations.  
Recall that the Cayley complex for a presentation of a group $G$ is the 
universal cover of the presentation 2-complex.  The group $G$ acts 
freely on its Cayley complex, with one orbit of vertices and with 
orbits of 1- and 2-cells corresponding to the generators and relators 
respectively in the presentation.  We define a partial Cayley complex
to be a $G$-invariant subcomplex of the Cayley complex; partial 
Cayley complexes are in bijective correspondence with subcomplexes of
the presentation complex.  

\begin{proposition}\label{prop:fp2}
Let $H$ be given by a presentation with finitely many 
generators and a countable set of relators $r_1,r_2,\ldots$.  
The following are equivalent.  
\begin{itemize} 
\item[(i)] $H$ is $FP_2$.  
\item[(ii)] There exists $m$ so that 
for each $i>m$, the loop defined by $r_i$ represents zero in 
the homology of the partial Cayley complex corresponding to 
all the generators and the relators $r_1,\ldots,r_m$.  
\item[(iii)] There is a connected free $H$-CW-complex with finitely 
many orbits of cells and perfect fundamental group.  
\end{itemize}
\end{proposition} 

\begin{proof} 
Equivalence of (i) and (ii).  Let $X$ be the Cayley complex for $H$
and let $X_m$ be the partial Cayley complex containing all 1-cells and 
only the 2-cells that correspond to the relators $r_1,\ldots,r_m$. 
Let $C_*(X)$ and $C_*(X_m)$ denote the cellular chain complexes of 
$X$ and $X_m$.  The image of the map
$d_1:C_1(X)\rightarrow C_0(X)$ is isomorphic to the augmentation 
ideal $I_H$.  Hence $H$ is
$FP_2$ if and only if the kernel of $d_1$ is finitely generated as a
$\zz H$-module.  Since $H_1(X)$ is trivial, this kernel is equal to
the image $d_2(C_2(X))$.  The stated condition on loops is equivalent
to $d_2(C_2(X_m))=d_2(C_2(X))$.  If this holds then clearly
$d_2(C_2(X))$ is finitely generated.  Conversely, any finite subset of
$d_2(C_2(X))$ is contained in some $d_2(C_2(X_m))$, so if
$d_2(C_2(X))$ is finitely generated then there exists $m$ with
$d_2(C_2(X_m))=d_2(C_2(X))$.

(ii)$\implies$(iii) and (iii)$\implies$(i).  Each $X_i$ is a connected
$H$-CW-complex with finitely many orbits of cells, and if (ii) holds
then $H_1(X_m)\cong H_1(X)$ is trivial.  Given any $H$-CW-complex $Y$
as in (iii), pick a maximal subtree $T$ in $Y/H$, let $\widetilde T$
be the set of lifts of $T$ in $Y$, and note that $\widetilde T$ is
equivariantly isomorphic to $T\times H$.  The cellular chain complex
$C_*(Y,\widetilde T)$ gives a finite presentation for the relative
homology group $H_1(Y,\widetilde T)$ as a $\zz H$-module.  Since
$H_1(Y)=0$, $H_1(Y,\widetilde T)$ is isomorphic to $I_H$.
\end{proof} 

Next we give the homological version of the Higman Rope 
Trick~\cite[IV.7.6]{lynsch}.  

\begin{lemma} \label{lem:higrope}
If $R$ is a homologically benign normal subgroup of a finitely
generated group $F$, then $F/R$ is embeddable in an $FP_2$ group.  
\end{lemma} 

\begin{proof} 
Fix $R$ as in the statement, and let $H$ be an $FP_2$ group containing
the group $F_R=\langle F, t\colon t^{-1}rt = r,\,\, r\in R\rangle$.
Let $L$ be the subgroup of $F_R\leq H$ generated by $F$ and
$t^{-1}Ft$, so that $L\cong F*_RF$.  As in \cite[IV.7.6]{lynsch} there
is a homomorphism $\phi:L\rightarrow F/R$ whose restriction to $F$ is
equal to the quotient map $F\rightarrow F/R$ and whose restriction to
$t^{-1}Ft$ is the trivial homomorphism.  Viewing $L$ as a subgroup of
$H$, the map $l\mapsto (l,\phi(l))$ defines a second copy of $L$
inside $H\times F/R$.  Let $K$ be the HNN-extension in which the
stable letter conjugates these two copies:
$$K= 
\langle H\times F/R, s\colon s^{-1}(l,1)s = (l,\phi(l)), \,\, l\in
L\rangle.$$ 
The group $K$ is generated by the generators for $H$, the generators 
for $F/R$ and the element~$s$.  As defining relators we may take the 
relators for $F/R$, the relators for $H$, finitely many relators 
stating that the generators for $H$ and the generators for $F/R$ 
commute, and finitely many relators of the form 
$s^{-1}(l,1)s(l,\phi(l))^{-1}$ for $l$ in some generating set for $L$.  
As in~\cite[IV.7.6]{lynsch}, the relators that hold between the 
generators for $F/R$ can be eliminated from this presentation for $K$, 
leaving just the relators for $H$ and finitely many other relators.  

To see that $K$ is $FP_2$, we use Proposition~\ref{prop:fp2} applied 
to the presentation 2-complex with the generators and relators 
described above.  The generators and relators for $H$ are contained 
in those for $K$, so we may look at the partial Cayley complex for 
$K$ corresponding to just these generators and relators.  This 
2-complex is isomorphic to a disjoint union of copies of the Cayley
complex for $H$ (one copy for each coset of $H$ in~$K$).  Let 
$r_1,r_2\ldots$ be an enumeration of the relators for $H$.  Since 
$H$ is $FP_2$, there exists $m$ so that for $i>m$, the relator $r_i$ 
represents zero in the homology of the partial Cayley complex for 
$H$ with just the relators $r_1,\ldots,r_m$.  It follows that these 
same loops represent zero in the homology of the partial Cayley 
complex for $K$ discussed above.  

Now consider the partial Cayley complex
for $K$, taking all the generators, the commutation relators between 
generators for $H$ and $F/R$, the finitely many relators involving 
$s$, and the relators $r_1,\ldots,r_m$.  For $i>m$, the loops in 
this complex defined by $r_i$ represent the zero element of homology, 
since they already represent 0 in the smaller partial Cayley complex 
consisting of a disjoint union of copies of the Cayley complex for
$H$.  Hence this presentation for $K$ satisfies condition~(ii) of 
Proposition~\ref{prop:fp2}, and so $K$ is $FP_2$.  
\end{proof} 

\begin{lemma}
Let $G$ be a finitely generated group which is embeddable in an $FP_2$
group.   
\begin{itemize} 
\item{} Every finitely generated subgroup of $G$ is homologically
  benign in $G$.   
\item{} If $H$ and $K$ are homologically benign subgroups of $G$, 
then so are their intersection and the subgroup that they generate.  
\end{itemize}
\end{lemma}

\begin{proof} 
Almost identical to the proof of \cite[Lemma~IV.7.7]{lynsch}, except 
that it relies on the fact that a free product with amalgamation 
$P=M*_GN$ is $FP_2$ provided that $M$ and $N$ are $FP_2$ and $G$ 
is finitely generated rather than on a similar statement for 
finite presentability.  This can be proved easily using 
Proposition~\ref{prop:fp2}. 
\end{proof} 

\begin{lemma} \label{lem:anys}
Fix $l\geq 4$, and for any integer $s$ define 
$v_s:= c_0^sc_1^s\cdots c_l^s d e^s$, an element of the free 
group $H=\langle c_0,\ldots,c_l, d, e\rangle$ of rank $l+3$.  
For any $S\subseteq \zz$ with $0\in S$, the subgroup 
$$V_S:=\langle v_s \colon s\in S\rangle 
\leq \langle c_0,\ldots,c_l, d, e\rangle$$ 
is homologically benign and is freely generated by the given
elements.  
\end{lemma} 

\begin{proof} 
If a reduced word in the elements $v_s$ is written out in terms of the
elements $c_0,\ldots,c_l,d,e$, the only cancellation that can take
place involves $c_0$ and $e$.  Thus the subwords $(c_1^sc_2^s\cdots
c_l^sd)^{\pm 1}$ survive uncancelled, which implies that the elements $v_s$ are
free generators for the subgroup $V_\zz$ of~$H$.

We claim that $V_\zz$ is benign, and hence homologically benign.  To
see this, define an ascending HNN-extension of the free group
$H=\langle c_0,\ldots,c_l,d,e\rangle$ by
\[ u^{-1}c_iu = c_0c_1\cdots c_{i-1}c_ic_{i-1}^{-1}c_{i-2}^{-1}\cdots c_0^{-1}, 
\quad
u^{-1}du = c_0c_1\cdots c_lde, \quad 
u^{-1}eu = e.  \]
Since $u^{-1}v_su=v_{s+1}$ for all $s\in \zz$ and $v_0=d$ it follows that 
$$\langle d,u\rangle \cap H = V_\zz.$$ 
Hence $V_\zz$ is benign in $\langle c_0,\ldots,c_l,d,e,u\rangle$ and
therefore also in the free group $H$.  

Fix some $S\subseteq \zz$, and claim that $V_S$ is homologically 
benign in $H$.  To see 
this, let $J=J(l,S)$ and $j_1,\ldots, j_l\in J$ be as in the statement
of Theorem~\ref{thm:fromufp}, and let $K=K(S)$ be 
$$K=\langle c_0,d,e\rangle * (\langle c_1,c_2,\ldots, c_l\rangle \times 
J) = H*_{\langle c_1,\ldots,c_l\rangle} 
(\langle c_1,c_2,\ldots, c_l\rangle \times J).$$
The group $K$ is $FP_2$, since it has a presentation in which the 
only relators are the relators of $J$ and finitely many commutation 
relators between $c_1,\ldots, c_l$ and the generators of $J$.  

Define an HNN-extension $M=M(S)$ of $K$, with base group $H$ and stable 
letter~$t$ via
$$t^{-1}c_0t=c_0,\quad t^{-1}c_it= c_ij_i\,\,\,\hbox{for $i>0$},\quad 
t^{-1}dt=d,\quad t^{-1}et=e.$$ 
The group $M$ is $FP_2$ and its subgroups $V_\zz$, $t^{-1}V_\zz t$ and $H$ 
are all homologically benign.  The elements $t^{-1}v_st$ freely generate 
the free group $t^{-1}V_\zz t$.  In terms of the generators for~$K$, $t^{-1}v_st = 
c_0^sc_1^s\cdots c_l^sj_1^s\cdots j_l^s de^s$.  When a reduced word in 
the elements $t^{-1}v_st$ is written in these terms, the 
only cancellation that can take place involves $c_0$ and $e$, thus the 
subwords $(c_1^s\cdots c_l^s j_1^s\cdots j_l^sd)^{\pm 1}$ survive 
uncancelled.  It follows that such a reduced word is in $H$ if and 
only if each subword $j_1^s\cdots j_l^s$ is equal to 1, or equivalently
each $s$ that occurs lies in $S$.  Hence 
$V_S$ is equal to $t^{-1}V_\zz t\cap H$ and is homologically benign in 
$M$ and in $H$.    
\end{proof} 

As in~\cite[IV.7]{lynsch}, let $L$ be the free group $L=\langle a,b\rangle$, 
and let $F$ be the free group of rank $l+6$ with $F=\langle 
a,b,c_0, \ldots,c_l, d,e,h\rangle$.  Define a G\"odel numbering $\gamma$ 
of {\sl all} words on the alphabet $\{a,b,a^{-1},b^{-1}\}$ by the formula 
$$\gamma(\emptyset)=0, \quad
\gamma(a)=1, \quad
\gamma(b)=2, \quad
\gamma(a^{-1})=3, \quad
\gamma(b^{-1})=4,$$
and extending to longer words by concatenation, viewing a concatenation 
of digits as a number.  Thus $\gamma$ is a bijection between 
the words and the subset of $\nn$ consisting of zero and all integers whose
decimal digits lie in the set $\{1,2,3,4\}$.  

To any word $w$ on $\{a,b,a^{-1},b^{-1}\}$, associate a codeword $g_w\in L$ 
defined by 
$$g_w:= whc_0^{\gamma(w)}c_1^{\gamma(w)}\cdots c_l^{\gamma(w)}de^{\gamma(w)}.$$ 
The subgroup $G$ of $F$ generated by all the elements $g_w$ is freely 
generated by them.  

\begin{lemma}
The subgroup $G$ is benign in $F$.   
\end{lemma} 

\begin{proof} 
Almost identical to the argument in~\cite[IV.7]{lynsch}.  Make a group 
$F^*$ defined as the fundamental group of a graph of 
groups with one vertex group $F$, and four edges corresponding to 
stable letters 
$u_\lambda$ for $\lambda\in\{a,b,a^{-1},b^{-1}\}$, each of which 
defines an ascending HNN-extension of $F$ with relations 
\[u_\lambda^{-1}au_\lambda = a, \quad
u_\lambda^{-1}bu_\lambda = b, \quad
u_\lambda^{-1}c_iu_\lambda= c_0^{\gamma(\lambda)}c_1^{\gamma(\lambda)} 
\cdots c_{i-1}^{\gamma(\lambda)}
c_i^{10}
c_{i-1}^{-\gamma(\lambda)}\cdots 
c_0^{-\gamma(\lambda)},\]
\[
u_\lambda^{-1}du_\lambda = c_0^{\gamma(\lambda)}c_1^{\gamma(\lambda)} \cdots
c_l^{\gamma(\lambda)}de^{\gamma(\lambda)},\quad 
u_\lambda^{-1}eu_\lambda = e^{10},\quad 
u_\lambda^{-1}hu_\lambda = \lambda h.
\]
In $F^*$, we have that for any word $w=\lambda_1\cdots \lambda_n$, 
$$u_{\lambda_1}^{-1}\cdots 
u_{\lambda_n}^{-1}g_\emptyset u_{\lambda_n}\cdots u_{\lambda_1}= 
u_{\lambda_1}^{-1}\cdots 
u_{\lambda_n}^{-1}hd u_{\lambda_n}\cdots u_{\lambda_1}= g_w,$$ 
and if $w=u\lambda$ then $u_\lambda g_w u_\lambda^{-1} = g_u$.  

To show that $G$ is benign in $F$, it suffices to show that in $F^*$, 
$$G=F\cap \langle g_\emptyset,u_a,u_b,u_{a^{-1}},u_{b^{-1}}\rangle.$$
From the equations given above, it is clear that the left-hand side is
contained in the right-hand side.  As in~\cite[IV.7]{lynsch}, to prove
the converse it suffices to show that whenever $z\in G$ and
$\lambda\in\{a,b,a^{-1},b^{-1}\}$ are such that $u_\lambda
zu_\lambda^{-1}\in F$, then in fact $u_\lambda z u_\lambda^{-1} \in G$, 
or equivalently $z\in u_\lambda^{-1} Gu_\lambda$.
For this, write $z=g_{w_1}^{\epsilon_1}\cdots g_{w_n}^{\epsilon_n}$ as a
reduced word in the elements $g_w$, with $\epsilon_i=\pm 1$.  When
this expression for $z$ is rewritten in terms of the generators for
$F$ and reduced, each subword of the form
$(c_1^{\gamma(w_i)}c_2^{\gamma(w_i)}\cdots c_l^{\gamma(w_i)}d)^{\epsilon_i}$ 
survives uncancelled, and any two 
such subwords are separated by a non-trivial word in the other 
generators $a,b,c_0,e,h$.  
Each of the natural free generators for $u_\lambda^{-1}Fu_\lambda$ except  
$u_\lambda^{-1}du_\lambda= c_0^{\gamma(\lambda)}c_1^{\gamma(\lambda)}\cdots 
c_l^{\gamma(\lambda)}de^{\gamma(\lambda)}$ has total exponent of each 
$c_i$ divisible by~10.  From this it follows that each $\gamma(w_i)$ 
is congruent to $\gamma(\lambda)$ modulo 10, and hence that 
$w_i = x_i\lambda$ for some shorter word $x_i$, so that $z\in 
u_\lambda^{-1}Gu_\lambda$ as required.  
\end{proof}

\begin{corollary}\label{cor:allbenign} 
Every subgroup of the free group $L=\langle a,b\rangle$ 
is homologically benign.  
\end{corollary} 

\begin{proof} 
Let $N$ be a subgroup of $L$, and define a subset $S=S(N)\subseteq \nn$ 
as the set of G\"odel codes for words $w$ on $\{a,b,a^{-1},b^{-1}\}$ that
are equal (as elements of $L$) to an element of~$N$: 
$$S=\{\gamma(w)\colon w \in_L N\}.$$  
Now let $Y_S$ be 
the free product $\langle a,b,h\rangle*V_S\leq F$, where $V_S$ is 
as defined in the statement of Lemma~\ref{lem:anys}.  By that lemma, 
$V_S$ is homologically benign, and hence $Y_S$ is homologically benign 
in $F$.  Since $Y_S$ is freely generated by $\{a,b,h,v_s\colon s\in S\}$, 
it is easy to see that $G\cap Y_S$ is freely generated by $\{g_w\colon
w\in N\}$.  (Recall that $v_s = c_0^sc_1^s\cdots c_l^sde^s$.)  Hence 
$G\cap Y_S$ is homologically benign.  The subgroup generated by 
$G \cap Y_S$ and the finite set $\{c_0,\ldots,c_l,d,e,h\}$, which 
is equal to $N*\langle c_0,\ldots,c_l,d,e,h\rangle$, is therefore also
homologically benign and the intersection of this group with $L$
is equal to $N$.  
\end{proof} 

We are now ready to complete the proof of Theorem~\ref{thm:main}.  By
the Higman-Neumann-Neumann embedding theorem~\cite{hnn,lynsch}, any countable
group can be embedded in a 2-generator group.  This 2-generator group
is isomorphic to $L/N$ for some normal subgroup $N$.  By
Corollary~\ref{cor:allbenign}, $N$ is homologically benign, and so by
Lemma~\ref{lem:higrope}, $L/N$ can be embedded in an $FP_2$ group.

\section{Closing remarks} 

An opinion attributed to Gromov~\cite[Ch.~1]{gdelah} is that any statement 
that is valid for every countable group should be trivial.  With this 
in mind, is there an easier, more direct proof of
Theorem~\ref{thm:main}?  Is there one that is not modelled on a proof 
of the Higman embedding theorem and that does not does not rely on 
Theorem~\ref{thm:fromufp}, or other results from~\cite{BB,ufp}?

To prove Theorem~\ref{thm:main}, we only need the groups $J(l,S)$ 
for some fixed $l\geq 4$.  Our motivation for allowing $l$ to vary 
comes from the above question.  For any $l\geq 4$ and any $S$ with $0\in
S\subseteq \zz$, define a group $J'(l,S)$ by the presentation 
\[J'(l,S) = \langle j_1,\ldots, j_l \colon j_1^sj_2^s\cdots j_l^s=1 \,\,\,
s\in S\rangle.\] 
If one could show that $J'(l,S)$ embeds in a group of type $FP_2$ 
and that $j_1^sj_2^s\cdots j_l^s\neq 1$ if $s\notin S$ without 
invoking~\cite{BB,ufp}, one would obtain a different proof of 
Theorem~\ref{thm:main}.  
If $l\geq 13$, the given presentation for $J'(l,S)$ satisfies 
the $C'(1/6)$ small cancellation condition~\cite[Ch.~5]{lynsch}.  
This can be used to give a different proof that 
$j_1^sj_2^s\cdots j_l^s\neq 1$ for $s\notin S$.  

The proof of the Higman-Neumann-Neumann embedding theorem 
in~\cite[IV.3]{lynsch} implies that any $FP_2$ group embeds 
in a 2-generator $FP_2$ group.  It follows that 
every countable group embeds in a 2-generator $FP_2$ group.  

The groups $J=J(l,S)$ in Theorem~\ref{thm:fromufp} may be chosen to have
cohomological dimension $\cd J=2$ in addition to the stated
properties.  By keeping track of the cohomological dimension at each
stage of the argument one obtains the following strengthened version of 
Corollary~\ref{cor:allbenign}, and hence a strengthened version of
Theorem~\ref{thm:main}: 

\begin{corollary}
For every subgroup $N$ of the free group $L=\langle a,b\rangle$, the 
HNN-extension $\langle L, t\colon t^{-1}nt= n \,\,\,n\in N\rangle$ 
embeds in an $FP_2$ group of cohomological dimension five.  
\end{corollary} 

\begin{theorem} \label{thm:maincd}
Every countable group $G$ embeds in a 2-generator $FP_2$ group $G^*$, with 
$\cd G^*\leq \cd G+5$.  Every torsion element in $G^*$ is conjugate to
an element of $G$.  
\end{theorem}  

The proof of the Higman embedding theorem in~\cite[IV.7]{lynsch} shows 
that every recursively presented group $G$ of finite cohomological
dimension embeds in a finitely presented group~$G^*$ of finite 
cohomological dimension.  However, $\cd G^*$ increases with the complexity 
of the Diophantine equation used to encode the relators in $G$.  Applying
Sapir's aspherical version of the Higman embedding theorem~\cite{sapir} 
gives the following.  
\begin{theorem}
For every recursive subgroup $N$ of the free group 
$L=\langle a,b\rangle$, the HNN-extension 
$\langle L, t\colon t^{-1}nt= n \,\,\,n\in N\rangle$ embeds in 
a finitely presented group of cohomological dimension two.  
\end{theorem} 
Combining this with the Higman rope trick~\cite[IV.7.6]{lynsch} 
gives a version of the Higman embedding theorem which 
is an analogue of~Theorem~\ref{thm:maincd}, but with a better bound on
$\cd G^*$.   

\begin{theorem} 
Every recursively presented group $G$ embeds into a finitely presented
2-generator group~$G^*$ with $\cd G^*\leq \cd G+2$.  Every torsion element in 
$G^*$ is conjugate to an element of $G$.  
\end{theorem}

\leftline{\bf Author's address:}

\obeylines

\smallskip
{\tt i.j.leary@soton.ac.uk}

\smallskip
School of Mathematical Sciences, 
University of Southampton, 
Southampton,
SO17 1BJ


\begin{thebibliography}{19} 

\bibitem{BB} M. Bestvina and N. Brady, \emph{Morse theory and
  finiteness properties of groups}, Invent. Math. \textbf{129} (1997), 
445--470.  


\bibitem{bieristrebel} R. Bieri and R. Strebel, \emph{Almost finitely
  presented soluble groups}, Comment. Math. Helvetici \textbf{53} 
(1978) 258--278.  

\bibitem{brown} K. S. Brown, \emph{Cohomology of Groups}, Grad. Texts
 in Math. \textbf{87}, Springer-Verlag (1982).  

\bibitem{gdelah} E. Ghys and P. de la Harpe, \emph{Panorama}, in the 
book \emph{Sur les Groupes Hyperboliques d'apr\`es 
Mikhael Gromov}, Progress in Mathematics \textbf{83}, Birkh\"auser
(1990) 1--25.  

\bibitem{higman} G. Higman, \emph{Subgroups of finitely presented
  groups}, Proc. Roy. Soc. Ser. A \textbf{262} (1961) 455--475.  

\bibitem{hnn} G. Higman, B. H. Neumann and H. Neumann, \emph{Embedding
  theorems for groups}, J. London Math. Soc. \textbf{24} (1949)
  247--254.  

\bibitem{ufp} I. J. Leary, \emph{Uncountably many groups of type
  $FP$}, to appear in Proc. London Math. Soc., 31 pages.  

\bibitem{lynsch} R. C. Lyndon and P. E. Schupp, \emph{Combinatorial
  Group Theory} (reprint of the 1977 edition), Classics in Mathematics.
Springer-Verlag (2001).  

\bibitem{neumann} B. H. Neumann, \emph{Some remarks on infinite
  groups}, Proc. London Math. Soc. \textbf{12} (1937) 120--127.  

\bibitem{sapir} M. Sapir, \emph{A Higman embedding preserving
  asphericity}, J. Amer Math. Soc. \textbf{27} (2014) 1--42.  

\bibitem{valiev} M. K. Valiev, \emph{One theorem of G. Higman},
  Algebra and Logic \textbf{7} (1968) 135--143.  

\end{thebibliography}
\end{document}